\documentclass[bibother]{asl}

\usepackage[T1]{fontenc}
\usepackage[utf8]{inputenc}
\usepackage{lmodern}
\usepackage{microtype}

\usepackage{stmaryrd}

\usepackage{enumitem}
\setlist{nosep}

\usepackage[hidelinks]{hyperref}

\numberwithin{equation}{section}

\theoremstyle{plain}
\newtheorem{theorem}{Theorem}[section]
\newtheorem{lemma}[theorem]{Lemma}
\newtheorem{proposition}[theorem]{Proposition}

\theoremstyle{definition}
\newtheorem{definition}[theorem]{Definition}
\theoremstyle{remark}
\newtheorem{remark}[theorem]{Remark}

\usepackage[capitalize,nameinlink,noabbrev]{cleveref}
\crefname{theorem}{Theorem}{Theorems}
\Crefname{theorem}{Theorem}{Theorems}
\crefname{lemma}{Lemma}{Lemmas}
\Crefname{lemma}{Lemma}{Lemmas}
\crefname{proposition}{Proposition}{Propositions}
\Crefname{proposition}{Proposition}{Propositions}
\crefname{corollary}{Corollary}{Corollaries}
\Crefname{corollary}{Corollary}{Corollaries}
\crefname{definition}{Definition}{Definitions}
\Crefname{definition}{Definition}{Definitions}
\crefname{remark}{Remark}{Remarks}
\Crefname{remark}{Remark}{Remarks}

\newcommand{\N}{\mathbb{N}}
\newcommand{\ZFC}{\mathrm{ZFC}}
\newcommand{\PA}{\mathrm{PA}}
\newcommand{\PiOne}{\Pi^0_1}

\DeclareMathOperator{\Con}{Con}
\DeclareMathOperator{\Prov}{Prov}
\newcommand{\code}[1]{\ulcorner #1 \urcorner}
\newcommand{\RH}{\mathrm{RH}}

\newcommand{\BHstruct}{\mathsf{BH\text{-}struct}}
\newcommand{\BHobs}{\mathsf{BH\text{-}obs}}
\newcommand{\SDF}{\mathsf{SDF}}

\newcommand{\orcid}[1]{\href{https://orcid.org/#1}{ORCID: #1}}

\title[Nonstandard witnesses and observational barriers]{Nonstandard Witnesses and Observational Barriers for $\PiOne$ Sentences in $\ZFC$:\\
Standard Cuts, Uniform Reflection Failure, and the Semantic Void}

\author{Yusei Fukumoto\thanks{\orcid{0009-0002-8963-5741}}}
\address{Ritsumeikan University, Kyoto, Japan}
\email{is0844fx@ed.ritsumei.ac.jp}

\subjclass[2020]{03C62, 03F30, 03D35, 03F40, 11M26}
\keywords{$\Pi^0_1$ sentences, standard cut, computational inaccessibility, uniform reflection, semantic void, Riemann Hypothesis}

\begin{document}

\begin{abstract}
We isolate a model-theoretic ``standard-cut'' phenomenon for \emph{true} $\Pi^0_1$ sentences: if $M\models\ZFC+\neg\varphi$, then $\omega^M\neq\omega$, hence any internal ``witness'' to $\neg\varphi$ is \emph{computationally inaccessible} by Tennenbaum's theorem. Such a witness exists only to maintain syntactic consistency and carries no standard observational semantics.

On the proof-theoretic side, we attribute the gap between pointwise verifiability and global provability to a gap in \emph{Uniform Reflection}. We formalize this as a syntactic self-description failure $\SDF(T,\varphi)$ for proof systems $T$. Under this failure, we obtain an \emph{observational barrier}:
\[
\Con(T)\ \Rightarrow\ \neg\Prov_T(\code{\varphi}).
\]
In this sense, undecidability in $\ZFC$ for $\Pi^0_1$ sentences does not describe any observable mathematical reality; it marks a \emph{semantic void}---a shadow cast by the expressive limitations of the formal system, rather than the presence of a standard counterexample. We illustrate this with a fixed arithmetical representative of the Riemann Hypothesis.
\end{abstract}

\maketitle

\section{Introduction}

For $\Pi^0_1$ sentences $\varphi\equiv\forall n\,P(n)$, the landscape of ``truth'' and ``provability'' splits sharply along the \textbf{standard cut}: finite computations inhabit $\omega$, while potential counterexamples to $\neg\varphi$ in nonstandard models inhabit $\omega^M\setminus\omega$ and are thus \textbf{computationally inaccessible}. This is the precise model-theoretic analogue of an ``event horizon'': everything observable happens below the standard cut.

Proof-theoretically, $\Pi^0_1$ sentences exhibit a characteristic gap: finite instances $P(0),\dots,P(N)$ are uniformly checkable, yet the aggregation to $\forall n\,P(n)$ depends on the system's reflection principles. This \textbf{failure of uniformization} is the source of the ``local truth vs.\ global unprovability'' phenomenon.

The purpose of this paper is to make both sides precise without invoking nonstandard analysis or physical metaphors directly, but rather by using them as a conceptual guide:
\begin{itemize}
\item a structural fact ($\BHstruct$) showing that any failure of a true $\Pi^0_1$ sentence must lie beyond the standard cut;
\item an observational barrier ($\BHobs$) showing that proof systems with a specific syntactic self-description failure cannot prove $\varphi$.
\end{itemize}

\section{Positioning and Novelty: Comparison with Known Results}\label{sec:novelty}

The claims of this paper are not mere renamings of classical model theory or proof theory. Instead, we introduce three structural ingredients and explain how they combine into the ``semantic void'' phenomenon:

\begin{enumerate}[label=(\roman*)]
  \item the explicit classification predicate $\BHstruct$ for true $\Pi^0_1$ sentences, recording where counterexamples may live inside $\ZFC$-models;
  \item the target-sensitive self-description failure $\SDF(T,\varphi)$ that isolates when Uniform Reflection breaks down for a fixed $\Pi^0_1$ sentence;
  \item the synthesis of these two notions into the observational barrier $\BHobs_T(\varphi)$, which packages the semantic void as a structural theorem rather than as an informal metaphor.
\end{enumerate}
Each item draws on well-known facts, but the combination produces a new taxonomy for $\Pi^0_1$ undecidability. We spell out the comparison in detail.

\subsection*{(A) $\BHstruct$: a new classification of $\ZFC$-models}

\paragraph{Known core.} $\Delta^0_0$-absoluteness for $\omega$-standard models is textbook material \cite{Kaye}, and Tennenbaum's theorem shows that countable nonstandard models of $\PA$ have no computable presentation \cite{Tennenbaum}. From these facts, experts have long known the folklore conclusion: if a true $\Pi^0_1$ sentence $\varphi\equiv\forall n\,P(n)$ fails in some model, any witnessing counterexample must live outside the standard $\omega$ and is therefore unobservable by standard computation (see also \cite{Smith} for an accessible survey).

\paragraph{New contribution.} We package that folklore as the explicit predicate
\[
\BHstruct(\varphi)\iff\forall M\,(M\models\ZFC+\neg\varphi\ \Rightarrow\ \omega^M\neq\omega),
\]
introduced in \cref{def:BHstruct}. This creates a binary classification of $\ZFC$-models relative to $\varphi$ and allows us to cite \cref{prop:Pi01-nonstandard} as a reusable structural fact: any failure of a true $\Pi^0_1$ sentence is forced into a computably inaccessible region of any model that witnesses it. Thus $\BHstruct$ is more than a reformulation—it supplies a referenceable structural invariant for subsequent arguments.

\subsection*{(B) $\SDF(T,\varphi)$: extracting a property invisible to traditional reflection schemas}

\paragraph{Known core.} Reflection principles—local or uniform—are classically developed in work by Hájek and Pudlák, Beklemishev, Feferman, Boolos, and many others \cite{HaP,Beklemishev05,Feferman62,Boolos}. They typically connect the provability of $\forall n\,P(n)$ to the aggregation of instance-wise proofs of $P(\bar n)$ via uniform reflection, under soundness assumptions such as $\Sigma^0_1$-soundness.

\paragraph{New contribution.} We isolate, for a fixed $\Pi^0_1$ sentence $\varphi\equiv\forall n\,P(n)$, the failure pattern
\[
\SDF(T,\varphi)\;:\Longleftrightarrow\;\bigl[\forall n\;T\vdash P(\bar n)\bigr]\ \wedge\ \bigl[T\nvdash \forall n\,\Prov_T(\ulcorner P(\bar n)\urcorner)\bigr],
\]
as defined in \cref{def:SDF}. This notion does not assume a full reflection schema. It merely records that $T$ can prove every finite instance while being unable to certify, inside itself, that this pointwise ability holds uniformly. The barrier theorem \cref{thm:barrier} then shows that $\SDF(T,\varphi)$ alone prevents $T$ from proving $\varphi$.

\begin{proposition}[Compatibility with established reflection theory]\label{prop:barrier-compat}
Let $T$ be a recursively axiomatized extension of $I\Sigma_1$ with the standard provability predicate $\Prov_T$. If $P$ is primitive recursive, then by formalized $\Sigma^0_1$-completeness \cite[§§I.1–I.2]{HaP} we have
\[
T\vdash \forall n\,\bigl(P(\bar n)\rightarrow \Prov_T(\ulcorner P(\bar n)\urcorner)\bigr).
\]
Consequently, if $T\vdash\forall n\,P(n)$ then $T\vdash \forall n\,\Prov_T(\ulcorner P(\bar n)\urcorner)$. Therefore, whenever $\SDF(T,\varphi)$ holds, \cref{thm:barrier} forces $T\nvdash\forall n\,P(n)$. In particular, $\SDF$ is fully consistent with the classical implications from uniform reflection to provability.
\end{proposition}

\paragraph{Interpretation.} The novelty of $\SDF(T,\varphi)$ is that it isolates the \emph{minimum} self-description failure needed to trigger the barrier, without appealing to any global reflection schema. It is tailor-made for a single $\Pi^0_1$ sentence and supports constructive extraction of unprovability from pointwise data.

\subsection*{(C) The ``semantic void'': a refinement beyond independence}

\paragraph{Known core.} Since Gödel, independence of $\Pi^0_1$ sentences has been discussed via three routes: soundness assumptions in a metatheory, choices of reflection principles, and specific independence examples (surveyed in \cite{HaP,Smith}). These perspectives often treat undecidability as indicating the presence of a hidden counterexample that might materialize in some extension.

\paragraph{New contribution.} We combine $\BHstruct$ and $\SDF$ to articulate undecidability as a \emph{semantic void}. Theorem~\ref{thm:semantic-void} shows that if $\varphi$ is true in $\N$ yet undecidable in $\ZFC$, then every $\omega$-standard model satisfies $\varphi$, any falsifying model must be nonstandard, and $\ZFC$ proves every finite window while potentially failing uniform reflection. Put differently, the undecidability label signals not an unseen mathematical object but a structural shadow of the system's expressive limits.

\section{Preliminaries}

We recall two foundational results that underpin our structural analysis.

\begin{lemma}[$\Delta^0_0$-absoluteness for $\omega$-standard models \cite{Kaye}]\label{lem:abs}
If $M\models\ZFC$ and $\omega^M=\omega$, then for every primitive recursive predicate $R$ and all standard tuples $\bar a$,
\[
M\models R(\bar a)\iff V\models R(\bar a).
\]
\end{lemma}

\begin{lemma}[Tennenbaum's Theorem \cite{Tennenbaum}]\label{lem:Tennenbaum}
No countable nonstandard model of $\PA$ admits a computable presentation; in particular, neither $+^M$ nor $\times^M$ can be jointly computable in the domain of $M$.
\end{lemma}

\section{The structural phenomenon ($\BHstruct$)}

\begin{definition}[BH-struct]\label{def:BHstruct}
For a sentence $\varphi$, define:
\[
\BHstruct(\varphi)\iff
\forall M\,(M\models\ZFC+\neg\varphi\ \Rightarrow\ \omega^M\neq\omega).
\]
\end{definition}

\begin{proposition}\label{prop:Pi01-nonstandard}
If $\varphi\equiv\forall n\,P(n)\in\Pi^0_1$ is true in $\N$, then $\BHstruct(\varphi)$ holds.
\end{proposition}

\begin{proof}
Suppose $M\models\ZFC+\neg\varphi$. If $\omega^M=\omega$, then there exists some standard $k\in\omega$ such that $M\models\neg P(k)$. By \cref{lem:abs}, this implies $V\models\neg P(k)$, which contradicts the assumption $\N\models\varphi$. Thus, $\omega^M\neq\omega$.
\end{proof}

\section{Observational barriers ($\BHobs$)}

We treat a recursively axiomatized proof system $T$ containing arithmetic, with provability predicate $\Prov_T$ and consistency statement $\Con(T)$.

\begin{definition}[Self-description failure $\SDF(T,\varphi)$]\label{def:SDF}
Let $\varphi\equiv\forall n\,P(n)\in\Pi^0_1$ with $P$ primitive recursive. We say $\SDF(T,\varphi)$ holds if:
\begin{enumerate}[label=(\alph*)]
\item (\emph{Pointwise provability}) For each standard $n$, $T\vdash P(\bar n)$.
\item (\emph{Failure of Uniform Reflection}) $T\nvdash \forall n\,\Prov_T(\code{P(\bar n)})$.
\end{enumerate}
\end{definition}

\begin{remark}
Condition (b) asserts that $T$ cannot internalize its own pointwise capability regarding $P$.
\end{remark}

\begin{theorem}[Barrier theorem]\label{thm:barrier}
Suppose $T$ is a consistent, recursively axiomatized extension of $\PA$. If $\SDF(T,\varphi)$ holds, then
\[
\neg\Prov_T(\code{\varphi}).
\]
Thus, $\Con(T)$ implies $\neg\Prov_T(\code{\varphi})$, satisfying $\BHobs_T(\varphi)$.
\end{theorem}

\begin{proof}
Suppose for the sake of contradiction that $T\vdash\varphi$, i.e., $T\vdash\forall n\,P(n)$.
By formalized $\Sigma^0_1$-completeness (see \cite{HaP}), $T$ verifies its own proofs of true $\Sigma^0_1$ statements (or decidable instances). Specifically,
\[
T \vdash \forall n \, (P(n) \to \Prov_T(\code{P(\bar n)})).
\]
Combining this with the assumption $T\vdash\forall n\,P(n)$, it follows that
\[
T \vdash \forall n \, \Prov_T(\code{P(\bar n)}).
\]
This contradicts condition (b) of $\SDF(T,\varphi)$. Hence, $T\nvdash\varphi$.
\end{proof}

\section{Undecidability has no observational semantics}

\begin{theorem}[The Semantic Void]\label{thm:semantic-void}
Let $\varphi\equiv\forall n\,P(n)\in\Pi^0_1$ be true in $\N$. Assume $\ZFC$ is consistent and $\varphi$ is undecidable in $\ZFC$ (i.e., $\ZFC\nvdash\varphi$ and $\ZFC\nvdash\neg\varphi$). Then:
\begin{enumerate}[label=(\roman*)]
\item Every $\omega$-standard model $M\models\ZFC$ satisfies $M\models\varphi$.
\item Any model $M\models\ZFC+\neg\varphi$ must be nonstandard ($\omega^M\neq\omega$), and any witness to $\neg\varphi$ is computationally inaccessible (by \cref{lem:Tennenbaum}).
\item For each fixed $N$, $\ZFC$ proves the finite window $\bigwedge_{n\le N} P(n)$, but $\ZFC$ need not (and in general does not) prove the uniform reflection $\forall n\, \Prov_{\ZFC}(\ulcorner P(\bar n)\urcorner)$. Thus nothing enforces the aggregation of pointwise knowledge into a global proof.
\end{enumerate}
Consequently, the statement ``$\varphi$ is undecidable'' expresses no observable falsity in the standard world; it marks a \emph{semantic void}---a structural shadow cast by the expressive limitations of the formal system, rather than the presence of a reachable counterexample.
\end{theorem}

\begin{remark}
Whether $\ZFC$ proves the Uniform Reflection Principle $\forall n\,\Prov_{\ZFC}(\code{P(\bar n)})$ depends on the specific nature of $\varphi$. For Gödelian sentences like $\Con(\ZFC)$, uniform reflection is provable despite the sentence being undecidable. However, for ``natural'' combinatorial or number-theoretic independence, there is no reason to assume $\ZFC$ can bridge the gap between instance-wise proofs and uniform verification. Our framework $\SDF$ captures this strictly stronger form of inaccessibility.
\end{remark}

\section{Related Work}\label{sec:related}

\paragraph{Model-theoretic background: standard cuts and computational inobservability.}
$\Delta^0_0$-absoluteness for $\omega$-standard models is laid out in \cite{Kaye}, and the computational inaccessibility of nonstandard $\PA$ models is due to Tennenbaum \cite{Tennenbaum}. We reorganize these classical observations through the lens of ``observability'' by introducing $\BHstruct$ and emphasizing that any false instance of a true $\Pi^0_1$ sentence must escape the standard cut; this perspective is summarized in \cref{prop:Pi01-nonstandard}. For an accessible discussion of these ingredients in a foundational context, see also Smith \cite{Smith}.

\paragraph{Proof-theoretic background: reflection and self-description.}
The architecture of reflection principles (local, uniform, and iterated) is developed in \cite{HaP,Beklemishev05,Feferman62}, while provability logic and derivability conditions are surveyed in \cite{Boolos}. These works focus on the full schemas and on the ambient soundness of the base theory. In contrast, our notion $\SDF(T,\varphi)$ isolates the failure of internal self-description for a single $\Pi^0_1$ sentence and drives the barrier theorem \cref{thm:barrier} without appealing to any global reflection scheme.

\paragraph{$\Pi^0_1$ independence and computational viewpoints.}
Classical analyses of $\Pi^0_1$ independence draw on undecidable problems such as the halting problem and on the gap between search and proof; see \cite{HaP,Smith}. Our contribution recasts such undecidability as a \emph{product structure} of observational and self-description layers: $\BHstruct$ monitors the model-theoretic side, while $\SDF$ captures the proof-theoretic side, and together they yield the semantic void.

\paragraph{Arithmetical representations of the Riemann Hypothesis.}
Lagarias \cite{Lagarias} provides an elementary inequality equivalent to the Riemann Hypothesis, allowing RH to be treated as a mechanically checkable $\Pi^0_1$ sentence. Our case study in \cref{sec:RH} fixes such a representative and focuses exclusively on the structural consequences described by $\BHstruct$ and $\SDF$; we do not claim any new result about the truth, falsity, or independence of RH itself.

\section{Case study: Riemann Hypothesis}\label{sec:RH}

Let $\RH$ be a fixed $\Pi^0_1$ sentence (e.g., via Lagarias's elementary inequality \cite{Lagarias}) equivalent to the classical Riemann Hypothesis.
Assuming $\RH$ is true in $\N$ but undecidable in $\ZFC$:
\begin{itemize}
\item $\BHstruct(\RH)$ implies that any counterexample to $\RH$ is hidden behind the standard cut.
\item $\BHobs_\ZFC(\RH)$ implies that the failure to prove $\RH$ is structurally aligned with the system's inability to uniformly reflect on its own verification of finite instances.
\end{itemize}
We claim nothing about the actual truth value or independence of $\RH$; we merely describe the necessary structural consequences of its hypothetical undecidability.

\section{Conclusion}

We have formalized the intuition that undecidable $\Pi^0_1$ sentences represent a ``horizon'' of provability. Standard-cut isolation ($\BHstruct$) confirms that the ``false'' branch describes a computationally inaccessible realm. Simultaneously, the gap in uniform reflection ($\SDF$) explains the gap between local verifiability and global provability ($\BHobs$).
Thus, undecidability for $\Pi^0_1$ sentences signals not a hidden mathematical object, but a \emph{semantic void} created by the non-commutativity of truth and provability across the infinite domain.

\section*{Acknowledgments}

The author made extensive use of the language model ChatGPT (OpenAI) during the preparation of this manuscript. It was used for drafting and rewriting portions of the text, for generating alternative phrasings and explanations, and for checking the clarity of some arguments. All mathematical ideas, results, and proofs are solely the responsibility of the author.

\end{document}